\documentclass[12pt]{amsart}
\usepackage{amsmath,amscd,amssymb,epsfig}
\usepackage{graphics,color}

\input epsf
\newtheorem{theorem}{Theorem}[section]
\newtheorem{lemma}[theorem]{Lemma}
\newtheorem{proposition}[theorem]{Proposition}
\newtheorem{corollary}[theorem]{Corollary}
\theoremstyle{definition}
\newtheorem{definition}[theorem]{Definition}

\theoremstyle{remark}

\newtheorem{observation}{Observation}

\newcommand{\ra}{\rightarrow}
\newcommand{\hra}{\hookrightarrow}

\providecommand{\inprod}[2]{\langle#1,#2\rangle}

\providecommand{\pbyp}[2]{\frac{\partial#1\,}{\partial#2\,}}

\newcommand{\Teich}{Teichm\"uller }

\def\ra{{\rightarrow}}

\def\bZ{{\mathbb Z}}
\def\bQ{{\mathbb Q}}

\def\mc{\mathcal}
\def\mf{\mathfrak}

\def\S1g{{\Sigma_{g,1}}}
\def\T1g{{\mathcal T}_{g,1}}

\def\I1g{{\mathcal I}_{g,1}}

\newcommand\Aut{\operatorname{Aut}}

\newcommand\Id{\operatorname{Id}}

\addtocounter{MaxMatrixCols}{5}

\newcommand{\mb}{\mathbf}

\begin{document}

\title[Groupoid Extensions of Mapping Class Representations]
{Groupoid Extensions of Mapping Class Representations for Bordered Surfaces}

\author{J{\o}rgen Ellegaard Andersen}
\address{Center for the Topology and Quantization of Moduli Spaces\\Department of Mathematics\\
Aarhus University\\
DK-8000 Aarhus C, Denmark\\}
\email{andersen{\char'100}imf.au.dk}

\author{Alex James  Bene}
\address{Department of Mathematics\\
University of Southern California\\
Los Angeles, CA 90089\\
USA\\
~{\rm and}~Center for the Topology and Quantization of Moduli Spaces\\Department of Mathematics\\
Aarhus University\\
DK-8000 Aarhus C, Denmark\\}
\email{bene{\char'100}math.usc.edu}

\author{R. C. Penner}
\address{Departments of Mathematics and Physics/Astronomy\\
University of Southern California\\
Los Angeles, CA 90089\\
USA\\
~{\rm and}~Center for the Topology and Quantization of Moduli Spaces\\Department of Mathematics\\
Aarhus University\\
DK-8000 Aarhus C, Denmark\\}
\email{rpenner{\char'100}math.usc.edu}

\thanks{AJB and RCP are happy to acknowledge useful conversations
with Suzuki Masaaki, AJB is appreciative of  conversations with Matthew Day,  and all authors likewise thank Takuya Sakasai and Jean-Baptiste Meilhan.}



\begin{abstract}
The mapping class group of a surface with one boundary component admits numerous
interesting representations including as a group of automorphisms of a free group and
as a group of symplectic transformations.   Insofar as the mapping class group can be
identified with the fundamental group of Riemann's moduli space, it
is furthermore identified with a subgroup of the fundamental path groupoid
upon choosing a basepoint.   A combinatorial model for this, the mapping class groupoid,  arises from
the invariant cell decomposition of Teichm\"uller space, whose fundamental path
groupoid is called the Ptolemy groupoid.  It is natural to try to extend representations of the
mapping class group to the  mapping class groupoid, i.e., construct a homomorphism from the mapping class groupoid to the same target that extends the given representations arising from various choices of basepoint.

Among others, we extend both aforementioned
representations to the groupoid level in this sense, where the
symplectic representation is lifted both rationally and integrally. 
The techniques of proof include several algorithms involving fatgraphs and chord diagrams.
The former extension is given by explicit formulae depending upon six essential cases,
and the kernel and image of the groupoid representation are computed.
Furthermore, this provides groupoid extensions of any representation of the
mapping class group that factors through its action on the
fundamental group of the surface including, for instance, the Magnus
representation and representations on the moduli spaces of flat
connections.
\end{abstract}

\maketitle

\section{Introduction}

Let $\S1g$ be a surface with genus $g\geq 1$ and one boundary
component, and let $\pi_1=\pi_1(\S1g,p)$ be its fundamental group
with respect to a basepoint $p$ lying on its
boundary $\partial\S1g$.   $\pi_1$ is non-canonically isomorphic to a free group
$F_{2g}$ on $2g$ generators, and the \emph{mapping class group}
$MC(\S1g)$ (i.e., the group of path components of the space of orientation-preserving
homeomorphisms fixing $\partial\S1g$  pointwise) 
acts on it in
a natural way.  In fact, it is a classical result \cite{nielsen} of
Nielsen that $MC(\S1g)$ can be identified with the subgroup of
$\Aut(\pi_1)$ which fixes the element 
of $\pi_1$ corresponding to $\partial\S1g$.

Following \cite{penner,penner98}, let us consider the Ptolemy
groupoid ${\mathfrak{Pt}}(\S1g)$, i.e., the combinatorial fundamental path
groupoid of \Teich space for $\S1g$,  where objects are suitable
equivalence classes of marked fatgraphs (trivalent except for one
univalent vertex, see the next section), and morphisms are given by
finite sequences of Whitehead moves connecting them (again, see the
next section).  In this way, any element of  $MC(\S1g)$ is
represented by a finite sequence of Whitehead moves starting from a
fixed trivalent fatgraph and ending on a combinatorially identical
fatgraph, where the sequence is uniquely determined up to known
relations.

Similarly, we define the mapping class groupoid $\mf{MC}(\S1g)$ and the Torelli groupoid ${\mathfrak{To}}(\S1g)$  to be the respective quotients of ${\mathfrak{Pt}}(\S1g)$ under the action of the mapping class group and the Torelli group $\mc{I}(\S1g)$ (i.e., the subgroup  of $MC(\S1g)$ acting trivially on the homology of $\S1g$).  Mapping classes are given by sequences of Whitehead moves beginning and ending at combinatorially identical fatgraphs, i.e., the same object of $\mf{MC}(\S1g)$, and elements of the Torelli group moreover preserve some, hence any, ``homology marking'' (as in \cite{moritapenner} and described at the end of Section \ref{sec:defns}).

By a \emph{groupoid representation}, we shall mean a map from a
groupoid to a group which respects
composition.  
 It is natural to ask whether  known representations of
the mapping class group $MC(\S1g)$ can be extended to representations
of $\mf{MC}(\S1g)$ or ${\mathfrak{Pt}}(\S1g)$, and in particular, one may wonder if  Nielsen's
embedding $N\colon MC(\S1g)\to Aut(F_{2g})$ extends to a groupoid 
representation.   In this paper (in Theorem \ref{thm:canAutlift}), we prove that the
answer is yes, and we give explicit formulae for our extension
$$\widehat N\colon  \mf{MC}(\S1g)\to Aut(F_{2g})$$ which are governed
by six essential cases of fatgraph combinatorics.    It is important
to remark  that Nielsen's embedding $N\colon MC(\S1g)\to Aut(F_{2g})$ is defined by the action of $MC(\S1g)$ on $\pi_1$ via an isomorphism $\pi_1\cong F_{2g}$ given by a  choice of generating set for $\pi_1$;  
 our
construction, 
 on the other hand, is canonical with target $Aut(F_{2g})$ and  relies on an algorithm which
canonically determines a generating set for $\pi_1(\S1g)$ by
constructing a maximal tree in each appropriate fatgraph (see the greedy algorithm in Section \ref{sec:greedy}).  The kernel and image of $\widehat{N}$ are computed
(in Propositions \ref{prop:kernel} and  \ref{image} respectively).
The automorphism group $Aut(\pi_1)$ acts on the
representation variety of $\pi_1$ in any group, hence so too do
$\mf{MC}(\S1g)$ and ${\mathfrak{Pt}}(\S1g)$.

It follows that representations of $MC(\S1g)$ which factor through the Nielsen embedding
$N\colon  MC(\S1g)\to Aut(\pi_1)$ also must extend to ${\mathfrak{Pt}}(\S1g)$.  In particular,
the Magnus representation (see Section 4)  $ MC({\S1g})\to Gl(2g,\bZ [\pi _1])$ extends to the groupoid
level
$${\mathfrak{Pt}}({\S1g})\to Gl(2g,\bZ [\pi _1]),$$
 and explicit formulae for this extension are also given.   The algorithm
here seems comparable in terms of complexity to existing algorithms \cite{morita,Suzuki02}
for the calculation of Magnus representations.

Utilizing further combinatorial algorithms, we obtain maps from the Ptolemy groupoid to various subgroups of $MC(\S1g)$ which can be considered as extensions of the appropriate   identity representations.  In particular, the extension $\widetilde{id}\colon {\mathfrak{Pt}}(\S1g)\ra MC(\S1g)$ of the identity representation of the mapping class group itself to the Ptolemy groupoid 
in Theorem \ref{thm:mclift} leads
 to a different representation ${\mathfrak{Pt}}(\S1g)\to Aut(\pi_1)$ as well as an 
  extension of  the symplectic representation $\tau _0\colon MC(\S1g )\to
Sp(H)\cong Sp(2g,\bZ)$  to a representation 
 $$\hat \tau _0\colon  \mf{MC}(\S1g) \to
Sp(2g, \bZ )$$ 
by explicit algorithms (in Corollary \ref{uglycor}).

As a general point, we remark that it is not surprising that these extensions exist, but rather that
they can be described fairly succinctly depending only upon six basic cases.  This same
feature will persist in other contexts as well, for instance in principle, an extension of the Meyer cocycle  \cite{meyer} to the groupoid
level should follow from the symplectic representation given here and further calculation.
We hope that the techniques of this paper might be generally useful in studying mapping class group representations.
See \cite{bamp} for extensions to the Ptolemy groupoid of the finite type invariants of 3-dimensional quantum topology, parts of which depend upon the algorithms developed here.
The extension of the present work to the setting of surfaces with several boundary components seems straight-forward, and we have restricted here to the case of surfaces with one
boundary component simply for convenience.

\section{Marked Bordered Fatgraphs}\label{sec:defns}

Given a graph $G$ (i.e., a finite connected 1-dimensional CW complex), let $\mathcal{E}_{or}(G)$ denote the set of oriented edges of $G$.  Given $\mb{e}\in\mathcal{E}_{or}(G)$,  let $\mb{\bar e}$ denote the same edge with the opposite orientation and let $v(\mb{e})$ denote the vertex to which $\mb{e}$ points.

A \emph{fatgraph} is a  graph  together with a  cyclic ordering of $\{ {\mb e}:v({\mb e})=v\}$ for each vertex $v$ of $G$.   This additional  structure gives rise to certain cyclically ordered sequences of oriented edges called the \emph{boundary cycles} of $G$, where  an oriented  edge $\mb{e}$   is followed by the next edge in the cyclic ordering at $v(\mb{e})$, but with the opposite orientation, so  that it points away from $v(\mb{e})$.  In depicting a fatgraph, we will always identify the cyclic ordering at a vertex with  the counterclockwise orientation of the plane, according to which we will represent the boundary cycle of $G$ as a path alongside it with $G$ on the left.    
 
 Any two consecutive oriented edges in the boundary cycle define a \emph{sector} of the fatgraph, and each sector $G$ can be associated to a unique vertex of $G$.   We say that a fatgraph $G$ with $n$ boundary cycles has \emph{genus} $g$ if its Euler characteristic is $\chi (G)=2-2g -n$.

An isomorphism between two fatgraphs is a bijection of edges and vertices which preserves the incidence relations of edges with vertices and the cyclic ordering at each vertex.  We shall always regard isomorphic fatgraphs as equivalent.

 A \emph{(once-)bordered} fatgraph is a fatgraph with only one boundary cycle such that all vertices are at least trivalent except for a unique univalent vertex.  A bordered fatgraph is ``rigid'' in the sense that any fatgraph automorphism is trivial.

There is a natural  linear ordering on the set $\mathcal{E}_{or}(G)$ of oriented edges of a bordered fatgraph $G$ obtained by setting $\mb{x}<\mb{y}$ if $\mb{x}$ appears before $\mb{y}$ while traversing the boundary cycle of $G$ beginning at the univalent vertex.    This provides each edge $e$ of $G$ with a preferred orientation, denoted simply by $\mb{e}\in\mathcal{E}_{or}(G)$, by requiring $\mb{e}<\mb{\bar e}$.  
  We call the edge incident to the univalent vertex  the \emph{tail} of $G$ and denote its preferred orientation by $\mb{t}$ so that $\mb{t}\leq \mb{x}$ for all $\mb{x}\in \mathcal{E}_{or}(G)$.

Given a trivalent bordered fatgraph $G$ and a non-tail edge $e$ of $G$, define the \emph{Whitehead move} on $e$ to be the collapse of $e$ followed by the unique distinct expansion of the resulting four-valent vertex.  (Any non-tail edge of $G$ necessarily has distinct endpoints since there is only one boundary cycle.)

There  is a natural composition on the set of Whitehead moves, where one Whitehead move $W\colon G_0\ra G_1$ can be composed with another $W'\colon G'_0\ra G'_1$ in the natural way if and only if $G_1=G'_0$.
\begin{definition}
As in \cite{penner,penner98}, the \emph{mapping class groupoid} $\mf{MC}(\S1g)$ of $\S1g$ is defined to be the set of  finite compositions of  Whitehead moves on  bordered fatgraphs  modulo the pentagon,  commutativity, and involutivity relations.
\end{definition}
$\mf{MC}(\S1g)$ can be identified with  the combinatorial  fundamental path groupoid of the dual cell decomposition of Riemann's moduli space of $\S 1g$ \cite{penner04}, and in this way, any element of the mapping class group $MC(\S1g )$ of $\S1g$ can be represented by a sequence of Whitehead moves  $\{W_i\colon G_{i-1}\ra G_i\}_{i=1}^k$ with $G_0=G_k$.

 \begin{definition}
 Fixing a point $q\neq p \in \partial\S1g$, a \emph{marking} of a   bordered fatgraph $G$ is an isotopy class of embeddings $f\colon G\hra \S1g$ such that the cyclic ordering at vertices of $G$ agrees
 with the  orientation of $\S1g$,  the complement $\S1g\backslash f(G)$ is contractible, and $f(G)\cap\partial\S1g=f(t)\cap\partial\S1g=\{q\}$.
 \end{definition}
 
Markings evolve unambiguously under Whitehead moves, and in this way, there is a natural composition on the set of  Whitehead moves acting on marked fatgraphs.

   \begin{definition}
Define the  \emph{Ptolemy groupoid} ${\mathfrak{Pt}}(\S1g )$ of $\S1g$ to be the set of finite sequences of composable Whitehead moves on genus $g$ marked bordered fatgraphs modulo the corresponding pentagon, commutativity, and involutivity relations, cf. \cite{moritapenner}.
  \end{definition}

As with the mapping class groupoid,  ${\mathfrak{Pt}}(\S1g )$ can be identified with a combinatorial version of the fundamental path groupoid of the \Teich space $\mathcal{T}_{g,1}$ of $\S1g$.  Since   $\mathcal{ T}_{g,1}$ is connected and simply connected, any two marked bordered fatgraphs are related  by a unique  element of  ${\mathfrak{Pt}}(\S1g)$, i.e., there is a sequence of Whitehead moves connecting the two which is uniquely determined modulo the pentagon, commutativity, and involutivity relations.

  The mapping class group $MC(\S1g )$  acts by post-composition on the set of markings of  $G$ in a  free and transitive manner, which directly corresponds to its free action on $\mathcal{ T}_{g,1}$.
  In this way, an element $\varphi$ of the mapping class group $MC(\S1g )$ is represented by any sequence of Whitehead moves 
  $\{W_i\colon (G_{i-1},f_{i-1})\ra (G_i,f_i)\}_{i=1}^k$ on marked fatgraphs $f_i:G_i\hra\S1g$, where
$(G_k,f_k)=(G_0, \varphi\circ f_0)$.

Fix a marking $f:G\hra\S1g$ of a fatgraph.  For each edge $e$ of $G$, there is a properly embedded ``dual'' arc, unique up to isotopy rel boundary, that meets $G$ only at a single transverse intersection point interior to $e$.  An orientation $\mb{e}$ on $e$ induces
an unambiguous orientation on its dual arc, where the pair of tangent vectors at the intersection point 
of the edge and arc in this order determine the positive orientation of the surface.  In this manner,
each marking of $G$ gives rise to a map $\pi_1\colon \mathcal{E}_{or}(G)\ra \pi_1$, which clearly satisfies the conditions of the next definition.

\begin{definition}
A \emph{geometric $\pi_1$-marking} of a  bordered fatgraph $G$ is a map $\pi_1\colon \mathcal{E}_{or}(G)\ra \pi_1$ which satisfies the following compatibility conditions:
\begin{itemize}
\item {\bf (edge)} we have
$\pi_1(\mb{e}) \pi_1(\mb{\bar e})=1$ for every oriented edge $\mb{e}\in \mathcal{E}_{or}(G)$;
\item {\bf (vertex)} we have
$
\pi_1(\mb{e}_1) \pi_1(\mb{ e}_2) \dotsm \pi_1(\mb{ e}_k)=1,
$
for every vertex $v$ of $G$, where $\mb{e}_1, \ldots, \mb{e}_k$ are the cyclically ordered oriented edges pointing towards $v$;
\item {\bf (surjectivity)} $\pi_1(\mathcal{E}_{or}(G))$ generates $\pi_1$;
\item {\bf (geometricity)} $
\pi_1(\mb{\bar t})~{\rm is~the~class~of~the~boundary}~\partial\S1g$.
\end{itemize}
\end{definition}

In fact, the two notions of marking are equivalent  \cite{bkp}, and we shall not distinguish between them in the sequel.  
Also for convenience, we shall henceforth denote $\pi_1(\mb{e})$ simply by $\mb{e}\in\pi_1$.

More generally, for any group $K$, we can define an \emph{abstract $K$-marking} of a fatgraph $G$ to be a map $\mathcal{E}_{or}(G)\ra K$ which satisfies the analogous edge and vertex  conditions, and we say that the $K$-marking is \emph{surjective} if the surjectivity condition is also satisfied.  By the compatibility conditions,  an abstract  $K$-marking evolves unambiguously under a Whitehead move, which moreover preserves surjectivity.

In particular, by composing a geometric $\pi_1$-marking with the abelianization homomorphism $\pi_1\ra H=H_1(\S1g,\bZ)$, one obtains what we call a \emph{geometric} $H$-marking of $G$,  which is a map $H\colon  \mathcal{E}_{or}(G)\ra H$ satisfying  the analogous abelian edge, vertex, and surjectivity conditions, as well as a geometricity condition which we now describe. This condition is expressed in terms of the 
skew pairing on $\mathcal{E}_{or}(G)$ given by 
$$\inprod{\mb{x}}{\mb{y}}=
\begin{cases}
-1,&~{\rm if}~\mb{x}<\mb{y}<\mb{\bar x}<\mb{\bar y};\\
\hskip 1.7ex 0,&~{\rm else};\\
+1,&~{\rm if}~\mb{x}<\mb{\bar y}<\mb{\bar x}<\mb{ y},\\
\end{cases}$$
where the conditions hold up to cyclic permutation along the boundary cycle, namely:

\begin{itemize}
\item {\bf ($H$-geometricity)} $\inprod{\mb{x}}{\mb{y}}=H({\mb x})\cdot H({\mb y})$ for all oriented edges $\mb{x},\mb{y}\in\mathcal{E}_{or}(G)$, where  $\cdot$ is the intersection pairing on $H$.
\end{itemize}

In fact,  a map $\mathcal{E}_{or}(G)\ra H$ is a geometric $H$-marking if and only if it satisfies the edge, vertex, surjectivity, and H-geometricity conditions \cite {bkp}.
Furthermore, $H$-markings evolve unambiguously under Whitehead moves and both the surjectivity and geometricity conditions are preserved under such moves.  

Following \cite{moritapenner}, we define the  \emph{Torelli groupoid} ${\mathfrak{To}}(\S1g )$ of $\S1g$ to be the set of finite sequences of Whitehead moves on geometrically $H$-marked genus $g$ bordered fatgraphs, together with the natural composition of sequences, modulo the corresponding pentagon, commutativity, and involutivity relations. 
The Torelli groupoid can be identified with the fundamental path groupoid of the Torelli cover of Riemann's moduli space corresponding to the kernel of the symplectic representation $\tau _0$, namely, the Torelli subgroup $\mc{I}(\S1g)$, again cf.  \cite{moritapenner}.

\section{The Greedy Algorithm}\label{sec:greedy}

In this section, we describe an algorithm for canonically determining a maximal tree in each bordered fatgraph.  

\begin{definition} {\bf (Greedy algorithm)} 
Define a subgraph $T_G$ of $G$ by $e\in T_G$ if $\mb{e}\leq \mb{x}$ for all  $\mb{x}\in\mathcal{E}_{or}(G)$ with $v(\mb{x})=v(\mb{e})$. 
 We call the linearly ordered set of oriented edges $\mb{X}_G=\{\mb{x}_i\}_{i=1}^{2g}$ determined by the complement $X_G=G\backslash T_G$ with its preferred orientations  the set of \emph{generators} for $G$.
\end{definition}

Note that there must  be at least  one and at most  two edges whose preferred orientations point  to a given trivalent vertex $v$, and these two cases correspond to whether the three sectors associated to $v$ are transversed in the counterclockwise or clockwise sense near $v$ along the boundary cycle.

\begin{lemma}
For each bordered fatgraph $G$, the subgraph $T_G$ is a maximal tree rooted by the tail of $G$.
\end{lemma}
\begin{proof}
Consider the following equivalent construction of the subgraph $T_G$.  Begin at the univalent vertex of $G$ and  traverse the boundary cycle of $G$ and ``greedily'' adding every edge to $T_G$ as long as the resulting subgraph is still a tree, meaning no non-trivial cycles would be  introduced.  Since the introduction of a non-trivial cycle from the addition of  an edge $e$ would mean the vertex $v(\mb{e})$  had previously been traversed,  this definition is equivalent to the original one.  From this perspective, $T_G$ is obviously a tree containing the tail, and it is maximal since adding any edge would result in a non-trivial cycle.
\end{proof}

\begin{theorem}
There is a canonical ordered set of generators of $\pi_1$ associated to every marked bordered fatgraph $G\hra \Sigma_{g,1}$. 
\end{theorem}
\begin{proof}
We take $\pi_1(\mb{X}_G)$ to be the desired set of generators of $\pi_1$ and need only show that they do indeed generate $\pi_1$.
Since a geometric $\pi_1$-marking satisfies the surjectivity condition, it suffices to show that for each oriented edge $\mb{e}$ of $G$ the element $\pi_1(\mb{e})$ is in the subgroup generated by  $\pi_1(\mb{X}_G)$.  To this end, note that each leaf $l$ of the tree $T_G$ is adjacent to two generators in $G$,   so by the vertex compatibility condition, the corresponding element $\pi_1(\mb{l})$ can be written as a product of two elements of $\pi_1(\mb{X}_G)$ (or their inverses).  The argument follows easily by induction.\hfill{}
\end{proof}

\begin{corollary}\label{cor:explicitiso}
For every marked bordered fatgraph, there is  an explicit canonical isomorphism $\pi_1\cong F_{2g}$.
\end{corollary}
\begin{proof}
 This follows immediately from the Hopfian property of $F_{2g}$.\hfill{}
\end{proof}

From now on when $G$ comes equipped with a marking, we shall  identify $\mb{X}_G$ with the ordered set $\pi_1(\mb{X}_G)$ of generators of $\pi_1$.

\vskip .2in

\begin{corollary}\label{cor:autlift}
To each Whitehead move $W\colon G\ra G'$ between marked trivalent bordered fatgraphs,  there is a canonically associated  element
$$\widetilde{N}(W)\in\Aut(\pi_1)$$
 which is natural in the sense that if $\{W_i\}$ is a sequence of Whitehead moves representing an element $\varphi\in MC(\S1g )\subset \Aut(\pi_1)$, then the composition of the $\widetilde{N}(W_i)$ agrees with the image $N(\varphi)$ of $\varphi$ under the Nielsen embedding.
\end{corollary}
\begin{proof}
Consider the isomorphism which maps the ordered generating set $\mb{X}_G$ to $\mb{X}_G'$.  Again by the Hopfian property of $\pi_1$, this is an automorphism, and it is obvious that it respects composition of Whitehead moves.  
 The last statement follows by noting that if the generating set for $(G,f)$ is $\pi_1(\mb{X}_G)$, then the generating set for $(G,\varphi \circ f)$ is $\varphi (\pi_1(\mb{X}_G))$ by construction.\hfill{}
\end{proof}

In the next section, we shall see that the representation  $\widetilde{N}$  can be described in fairly concrete terms.  Moreover in Section \ref{sect:kernel}, we shall explicitly describe the kernel  of $\widetilde{N}$ (see Proposition \ref{prop:kernel}),  and in Section \ref{sect:image}, we shall describe the image of $\widetilde{N}$ (see Proposition \ref{image}).

\begin{figure}[!h]
\begin{center}
\epsffile{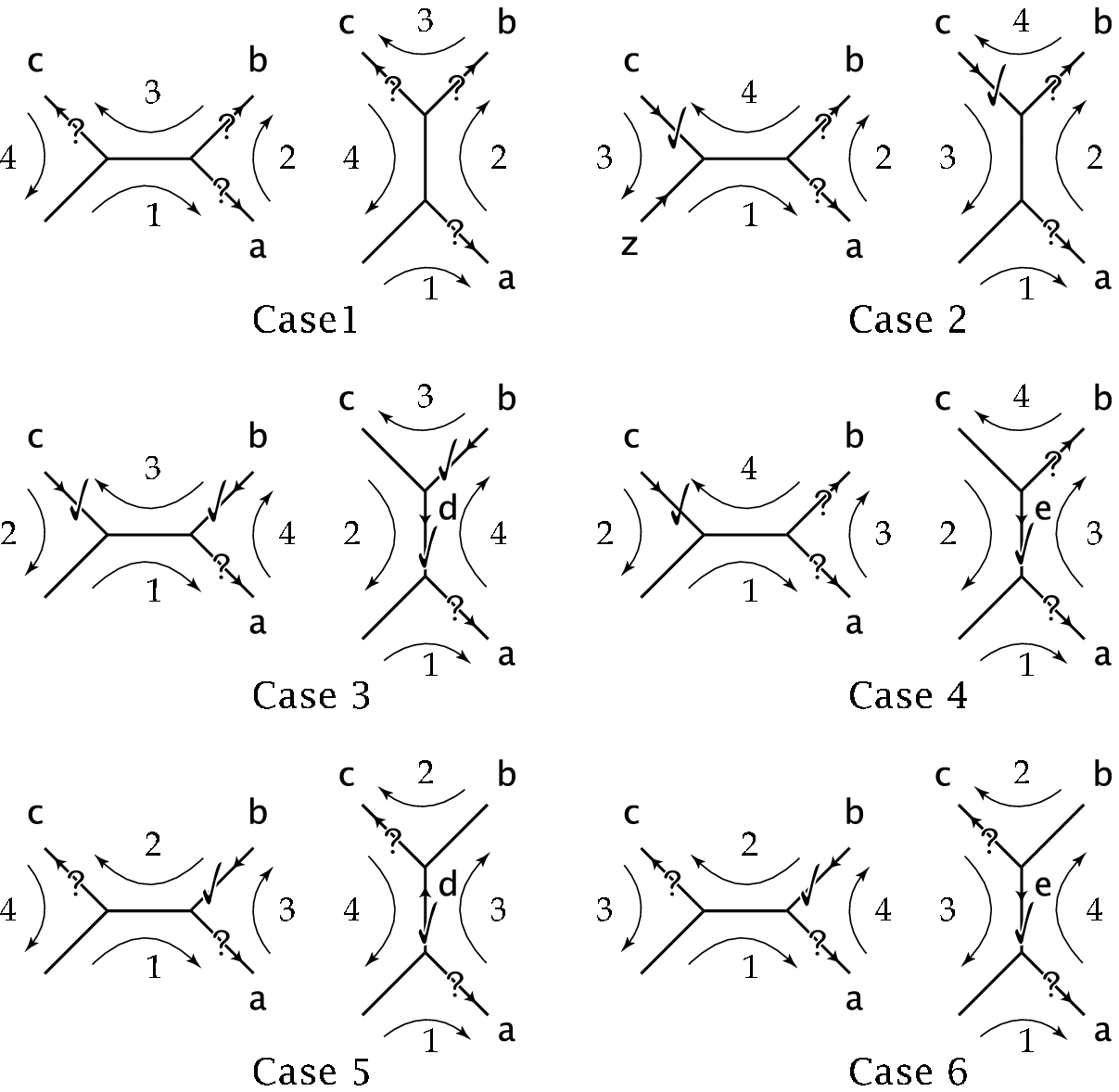}
\caption{Six cases of Whitehead moves $W\colon G\ra G'$}
\label{fig:GW1}
\end{center}
\end{figure}

\subsection{Essential cases of $\widetilde{N}(W)$}

Whitehead moves on bordered fatgraphs can be categorized into six basic types determined by the order of traversal of nearby sectors in the boundary cycle as
depicted in Figure \ref{fig:GW1}, and we now turn towards calculating $\widetilde{N}(W)$  for each.   We  say that  a Whitehead move $W\colon G\ra G'$ is a \emph{type $k$} move if it or its inverse corresponds the the $k$th case  according to our labeling in the figure.  We will find it most illuminating  to write our expressions  as elements of $Aut(F_{2g})$ rather than $Aut(\pi_1)$ via the isomorphism $\pi_1\cong F_{2g}$ provided by Corollary \ref{cor:explicitiso}.  

 First, consider the type 1 Whitehead move.  The initial fatgraph  $G$ has three edges $\mb{a}$, $\mb{b}$, and $\mb{c}$ which may  be generators (represented by question marks) depending on the global properties of the graph (not depicted).  The resulting fatgraph $G'$ similarly has three possible generators which are naturally identified with those of the first fatgraph. By construction,  $\mb{a}$ is a generator of $G$ if and only if it is a generator for $G'$ and similarly for the edges $\mb{b}$ and $\mb{c}$.  Moreover, the order of appearance of these generators in $\mb{X}_G$ and $\mb{X}_{G'}$ must be the same.  Thus, the element of $\Aut(F_{2g})$ corresponding to this Whitehead move is the identity element.

For a type 2  move, whereas the edges $\mb{a}$ and $\mb{b}$ perhaps may not be generators,
the edge 
$\mb{c}$ is definitely a generator (represented by a check mark) since 
 the vertex to which it points was first traversed in sector 1.
 In any case,   the corresponding element of $\Aut(F_{2g})$  is again the identity element.

Next,  consider a type 3 Whitehead move; note that the edges $a$ and $c$ may coincide. In any case, the edges $\mb{b}$ and  $\mb{c}$ must be generators of $G$ while $\mb{b}$ and $\mb{d}$ must be generators of $G'$.  Moreover, if $\mb{c}$ is the $i$th generator $\mb{x}_i$ of $G$, then $\mb{d}$ must be the $i$th generator of $G'$ so that   under the Whitehead move we have $\mb{c}\mapsto \mb{d}$ while all other generators are fixed.  Now note that by the vertex condition for $G'$, we have the relation $\mb{b}\mb{c}\mb{\bar d}=1$ so that  $\mb{c}\mapsto \mb{d}=\mb{ b}\mb{c}$.  If $\mb{b}$ is the $j$th generator $\mb{x}_j$ of $G$,   then we can explicitly write the corresponding  element of $\Aut(F_{2g})$ as
\[
\begin{array}{ll}
\mb{x}_k  \mapsto \mb{x}_k, &\textrm{ for } k\neq i,\\
\mb{x}_i \mapsto \mb{ x}_j \mb{x}_i.
\end{array}
\]

For case 4, the situation is almost identical to case 3 except that now $\mb{b}$ need not be a generator, and we have the slightly different relation $\mb{c}\mapsto \mb{e}=\mb{\bar b}\mb{c}$.  If $\mb{b}$ is a generator, say $\mb{x}_j$, then we find
\[
\begin{array}{ll}
\mb{x}_k  \mapsto \mb{x}_k,  &\textrm{ for } k\neq i,\\
\mb{x}_i \mapsto \mb{\bar  x}_j \mb{x}_i.
\end{array}
\]
If $\mb{b}$ is not a generator of $G$, then we must first express $\mb{b}$ as a word in the generators (which can be obtained from the combinatorics of the fatgraph) before arriving at an explicit element of $\Aut(F_{2g})$.

Consider now case 5, where  the edge $\mb{b}$ must be a generator of $G$ while the edge $\mb{d}$ must be a generator of $G'$.  The vertex condition forces the relation $\mb{d}\mb{b}\mb{\bar c}=1$, so that $\mb{d}=\mb{c}\mb{\bar b}$.  Now assume that $\mb{b}$ is the $i$th generator of $G$ so that $X_G=(\mb{x}_1, \ldots, \mb{x}_{i-1}, \mb{b}, \mb{x}_{i+1}, \ldots, \mb{x}_{2g})$  and that $\mb{d}$ is the $j$th generator of $G'$.  Under this Whitehead move, we find that
\[
\begin{array}{ll}
\mb{x}_k  \mapsto \mb{x}_k,  &\textrm{ for } k<i, \\
\mb{x}_k  \mapsto \mb{x}_{k+1},  &\textrm{ for } i\leq k<j ,\\
\mb{x}_j  \mapsto \mb{c}\mb{\bar x}_i ,& \\
\mb{x}_k  \mapsto \mb{x}_k,  &\textrm{ for } k>j.
\end{array}
\]
If $\mb{c}$ is a generator of $G$ (so that  $\mb{c}=\mb{ x}_{i+1}$), then the above maps explicitly determine the element of $\Aut(F_{2g})$, and otherwise, one must first express $\mb{c}$ as a word in the $\mb{x}_k$.

For the Whitehead move of type 6,  we have a situation which is essentially identical to that of case 5 except that now the generator $\mb{e}$ has an orientation which is opposite that of the generator $\mb{d}$ of case 5.  Thus, if we let $\mb{e}$ be the $j$th generator of $G'$, then we get the same mapping $X_G\mapsto X_{G'}$ as in case 5 except that $\mb{x}_j\mapsto \mb{x}_i\mb{\bar c}$.

 Thus, the  values respectively taken by our representation  $\widetilde{N}$ for the six essential  types of  Whitehead moves  are the identity in the first two cases, ``local'' in the third case in the sense that $\widetilde{N}(W)$ depends only upon the edges near the edge of the Whitehead move, and not necessarily local in the remaining cases.
 We can summarize our results with  the following
 
\begin{theorem}\label{thm:canAutlift}
There is an explicit extension
\[
\widehat{N}\colon \mf{MC}(\S1g)\ra Aut(F_{2g})
\]
of Nielsen's embedding to a representation of 
the mapping class groupoid with target $Aut(F_{2g})$.  Its value  $\widehat{N}(W)$ for a Whitehead move $W\colon G\ra G'$ on an edge $e$ of $G$ is explicitly calculable, and its 
particular
form depends on six essential cases corresponding to the possible  orders of traversal of the four sectors surrounding the edge $e$.
\end{theorem}
\begin{proof}
 Since the formulae for the representation $\widetilde{N}$ in terms of  $Aut(F_{2g})$ did not depend on the explicit markings of the fatgraphs, they define a map $\widehat{N} \colon \mf{MC}(\S1g) \ra Aut(F_{2g})$, which we claim is 
 a representation in the sense that
for any two composable Whitehead moves $W_1\colon G\ra G_1$ and $W_2\colon G_1\ra G_2$, 
we have $\widehat{N}(W_1\circ W_2)=\widehat{N}(W_2)\circ\widehat{N}(W_1)$; this change of ordering reflects the simple change of composition for functions from right-to-left and for concatenation of paths from left-to-right.
Indeed, this follows from the fact that for elements $\varphi,\psi\in Aut(F_{2g})$ defined  by $\varphi:\mb{x}_i\mapsto\mb{u}_i=\mb{u}_i(\mb{x}_1,\ldots, \mb{x}_{2g})$ and $\psi:\mb{u}_i\mapsto \mb{w}_i=\mb{w}_i(\mb{u}_1,\ldots, \mb{u}_{2g})$, the composition $\psi \varphi=\varphi(\varphi^{-1} \psi \varphi)\in Aut(F_{2g})$ is given by 
\[
\mb{x}_i\mapsto \mb{w}_i=\mb{w}_i(\mb{x}_1,\ldots, \mb{x}_{2g})=\mb{w}_i(\mb{u}_1,\ldots, \mb{u}_{2g})|_{\mb{u}_i=\mb{u}_i(\mb{x}_1,\ldots, \mb{x}_{2g})}.
\]\hfill{}
\end{proof}

\section{The Magnus representation}

Recall \cite{fox}  that the Fox free derivative with respect to $\mb{x}_i$  can be defined as the unique derivation $\pbyp{}{\mb{x}_i}\colon \bZ[ \pi_1] \ra\bZ [\pi_1]$ satisfying $\pbyp{}{\mb{x}_i}(\mb{x}_i)=1$, $\pbyp{}{\mb{x}_i}(\mb{\bar x}_i)=-\mb{\bar x}_i$, and the product rule $\pbyp{}{\mb{x}_i}(\mb{w}_1 \mb{w}_2)=\pbyp{}{\mb{x}_i}(\mb{w}_1)+\mb{w}_1\pbyp{}{\mb{x}_i}(\mb{w}_2)$.   One of the important properties of the Fox free derivative is the chain rule which states that if  $\mb{u}_1,\ldots, \mb{u}_{2g}$ is another generating set for $\pi_1$ and $\mb{w}\in\pi_1$ is a word, then
\begin{equation}\label{eq:chainrule}
\pbyp{\mb{w}}{\mb{x}_j}=\sum_{i=1}^{2g}\left( \pbyp{\mb{w}}{\mb{u}_i} \right)_{\mb{u}_i=\mb{u}_i(\mb{x}_1,\ldots, \mb{x}_{2g})}  \left( \pbyp{\mb{u}_i}{\mb{x}_j} \right).
\end{equation}
The classical Magnus representation of $\Aut(\pi_1)$ is the map which associates to any element $\varphi$ of  $\Aut(\pi_1)$ its Fox Jacobian $\left( \pbyp{\varphi(\mb{x}_i)}{\mb{x}_j}\right)$ with respect to a given  basis $\{\mb{x}_i\}_{i=1}^{2g}$;  this map  is a  crossed homomorphism by \eqref{eq:chainrule}, cf. \cite{morita}.

It is an immediate consequence of Corollary \ref{cor:autlift} that the Magnus representation extends to the Ptolemy groupoid.  However, such an extension a priori would be non-canonical as it would depend on a choice of generating set for $\pi_1$.  Instead, we extend the Magnus representation by 
\[
\widetilde{M}(W\colon G\ra G')=\left( \pbyp{x_i'}{x_j} \right)  
\] 
where $\{\mb{x}_i\}_{i=1}^{2g}$ and $\{\mb{x}'_i\}_{i=1}^{2g}$ are the sets of generators for $G$ and $G'$ respectively.  As a consequence of this definition and \eqref{eq:chainrule}, we have the 

\begin{corollary}
The Magnus representation explicitly extends to a representation $\widetilde M$ of  the Ptolemy groupoid with target $Gl(2g,\bZ[\pi_1])$.
\end{corollary}

Again, the formulae are governed by the six types of Whitehead moves, and we proceed to describe each.
The first non-trivial type is the third, where we have $\mb{x}_i \mapsto \mb{ x}_j \mb{x}_i$ and find a matrix in $Gl(2g,\bZ[\pi_1])$ which is the identity except for the $i$th  row
\begin{equation} \label{W2mtrx}
\left(  0, \ldots,  \mb{ x}_j, \ldots,   1,\ldots , 0 \right),
\end{equation}
which has all entries zero except for $\pbyp{}{\mb{x}_i}( \mb{ x}_j\mb{x}_i )=\mb{x}_j$ in the $i$th position and $\pbyp{}{\mb{x}_j}( \mb{ x}_j\mb{x}_i )=1$ in the $j$th position.

In case 4, we have $\mb{x}_i \mapsto \mb{\bar  b} \mb{x}_i$,  which again gives a  matrix differing from the identity only in its $i$th row, where If  $\mb{b}=\mb{x}_j$ is a generator, then this row
\begin{equation}\label{W2invmtrx}
\left(  0, \ldots,  \mb{\bar x}_j, \ldots,   -1,\ldots , 0 \right)
\end{equation}
has all entries zero except for $\mb{\bar x}_j$ in the $i$th position and $-1$ in the $j$th position.  If $\mb{b}$ is not a generator, there is a more complicated matrix with the $i$th row the Fox gradient:
\[
\left(\pbyp{\mb{\bar b }}{\mb{x}_1},  \pbyp{\mb{\bar b }}{\mb{x}_2}, \ldots , \mb{\bar b}, \pbyp{\mb{\bar b }}{\mb{x}_{i+1}} , \ldots    ,\pbyp{\mb{\bar b }}{\mb{x}_{2g}} \right).
\]

Assuming that $\mb{c}=\mb{x}_{i+1}$ is a generator in case 5, the corresponding matrix is the identity except for the  $(i,j)$ submatrix which is given by
\[
\begin{pmatrix}
0& 1 &  & \ldots &0\\
& 0 & 1 & \\
&&0&\ddots \\
&&&\ddots&1\\
-\mb{x}_{i+1}\mb{\bar x}_i& 1& 0 &\ldots&0
\end{pmatrix}.
\]
If $\mb{c}$ is not a generator, then the $j$th row is replaced by
\[
\left(
\pbyp{\mb{c}}{\mb{x}_{1}},\pbyp{\mb{c}}{\mb{x}_2} ,\ldots,\pbyp{\mb{c}}{\mb{x}_{i-1}} , \pbyp{\mb{c}}{\mb{x}_{i}} -\mb{c \bar{x}_i} , \pbyp{\mb{c}}{\mb{x}_{i+1}} ,\ldots  , \pbyp{\mb{c}}{\mb{x}_{2g}}
\right).
\]

Case 6 is almost identical to case 5, except that in this case if $\mb{c}$ is not a generator, then the $j$th row is replaced by
\[
\left(
\pbyp{\mb{\bar x}_i\mb{c}}{\mb{x}_{1}},\pbyp{\mb{\bar x}_i\mb{c}}{\mb{x}_2} ,\ldots,\pbyp{\mb{\bar x}_i\mb{c}}{\mb{x}_{i-1}} , -\mb{\bar x}_i + \mb{\bar x}_i\pbyp{\mb{c}}{\mb{x}_{i}} , \pbyp{\mb{\bar x}_i\mb{c}}{\mb{x}_{i+1}} ,\ldots  , \pbyp{\mb{\bar x}_i\mb{c}}{\mb{x}_{2g}}
\right).
\]
This completes the discussion of the various cases.

Morita \cite{morita} introduced variations $M_k\colon MC(\S1g) \ra Gl(2g, \bZ[N_k])$  of the Magnus representation by composing the classical Magnus representation described above with the quotient maps $\pi_1\ra N_k$, where $N_k=\pi_i/\pi_1^{(k)}$ is the $k$th nilpotent quotient of $\pi_1$ (for $k=1$, see \cite {Suzuki02}).  In the same way, our extension of the Magnus representation immediately yields extensions
$$\widetilde{M}_k\colon {\mathfrak{Pt}}(\S1g)\ra Gl(2g, \bZ[N_k]).$$
  Moreover, the value of these extensions on a Whitehead move $W\colon G \ra G'$ can be computed purely from the combinatorics of $G$ together with the surjective $N_k$-markings of $G$  induced from its $\pi_1$-marking.  In particular in the case $k=1$, we obtain a representation $\widetilde{M}_1\colon {\mathfrak{Pt}}(\S1g)\ra Gl(2g,\bZ[H])$ whose value on $W\colon G\ra G'$ depends only on the $H$-marking of $G$.  Thus,  we also have the stronger result:
\begin{proposition}
The representation $M_1\colon MC(\S1g)\ra Gl(2g,\bZ[H])$ extends to a representation  of  the Torelli groupoid  $$\widetilde M_H\colon {\mathfrak{To}}(\S1g)\ra Gl(2g,\bZ[H]).$$
\end{proposition}

\section{Linear chord diagrams and the kernel of $\widetilde{N}$ }\label{sect:kernel}
We determine the kernel  of the extension $\widetilde N \colon {\mathfrak{Pt}}(\S1g)\ra \Aut(\pi_1)$
of the Nielsen embedding in this section.

\begin{lemma}\label{lem:chord}
Given any trivalent marked bordered fatgraph $G=G_0$, there is a sequence of Whitehead moves $\{W_i\colon G_{i-1}\ra G_i\}_{i=1}^{k}$ with $\widetilde{N}(W_i)=\Id\in \Aut(\pi_1)$, for all $i$, such that
$G_k$ is a fatgraph whose maximal tree $T_{G_k}$ is a line segment.
\end{lemma}
\begin{proof}
Let $S_G \subset T_G$ be the subtree of $T_G$ defined by $s\in S_G$ if and only if $\mb{s}<\mb{x}$ for all $\mb{x}\in \mb{X}_G$, so $S_G$ is a line segment.  If $S_G=T_G$, then we are done, so assume otherwise.  Since $T_G$ is connected, there is an $e\in T_G-S_G$ which is adjacent to two edges of $S_G$, and since $e$ is in $T_G$, $\mb{e}$ must point away from $S_G$.  One can check that this dictates that the boundary cycle first traverses the sector containing $e$ so that that $\mb{e}$ points away from it and next traverses the sector to the right of $\mb{e}$.  As a result, the Whitehead move $W_e$ on $e$ must be  a move of type $1$ or $2$ so that $\widetilde{N}(W_e)=\Id$.
Moreover, under the move $W_e$, the length of  $S_G$ is increased by one.  By repeated application of this process, we obtain the desired sequence of moves resulting in a fatgraph $G_k$ with $T_{G_k}=S_{G_k}$.\hfill{}
\end{proof}

We let $C_G$ denote the fatgraph resulting from this procedure, which is called the {\it branch reduction algorithm}.

Recall from \cite{barnatan} that a \emph{linear chord} diagram is a segment in the real  line, called the \emph{core} of the diagram,  together with a collection of arcs, called the \emph{chords}, with endpoints attached to the core at distinct points. By identifying the subgraph  $T_{C_G}$ of $C_G$ with a portion of the real line, we see that the set $X_G$ can be viewed as a collection of chords attached to this core.  Strictly speaking however, this results in a diagram with the two right-most chords attached to the same point;  thus, in order to obtain a true chord diagram, we  add a  bivalent vertex to the right-most chord in $C_G$ and consider its     first half  as part of the core.  See Figure \ref{fig:obs}.

\begin{figure}[!h]
\begin{center}
\epsffile{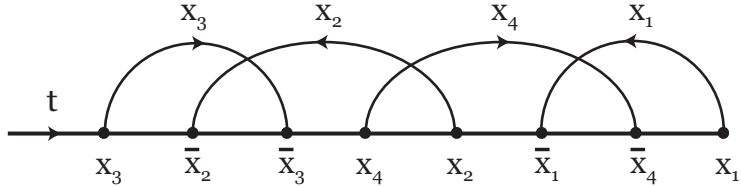}
\caption{Illustration of Observation \ref{word}.}
\label{fig:obs}
\end{center}
\end{figure}

We now make two observations:
\begin{observation}\label{word}  By repeated application of the orientation and vertex conditions, the word representing $\mb{t}$ in the letters $\mb{X}_G$  can be computed directly  from the chord diagram $C_G$.  Namely, by associating the element $\mb{\bar x}_i$ (respectively $\mb{ x}_i$) to the vertex $v(\mb{x}_i)$ (respectively $v(\mb{\bar x}_i)$), $\mb{t}$ is obtained by simply multiplying these elements in their left-to-right ordering along the core of $C_G$.  For example in Figure \ref{fig:obs}, we have $\mb{t}=\mb{x}_3\mb{\bar x}_2\mb{\bar x}_3\mb{x}_4\mb{x}_2\mb{\bar x}_1\mb{\bar x}_4\mb{x}_1$.
 \end{observation}
\begin{observation}
The word representing $\mb{t}$ obtained in this way is reduced since the fatgraph $C_G$ has only one boundary cycle.
\end{observation}

\begin{lemma}
The (marked) fatgraph $C_G$ obtained by the algorithm of Lemma \ref{lem:chord} is well-defined in the sense that if  $\mb{X}_G=\mb{X}_{G'}$, then $C_G=C_{G'}$.
\end{lemma}
\begin{proof}
This follows from the above observations
since there is a unique reduced word representing any element of a free group with respect to a given set of generators.\hfill{}
\end{proof}

As a result of the previous two  lemmas, we have the following
\begin{proposition}\label{prop:kernel}
The kernel of the extension of $\widetilde{N}\colon {\mathfrak{Pt}}(\S1g)\ra \Aut(\pi_1)$ is generated by type 1 and type 2 moves, i.e., any element in the kernel of  $\widetilde{N}$ is equivalent under pentagon, commutativity, and involutivity relations to a composition of type 1 and 2 moves.
\end{proposition}
\begin{proof}
Consider  any sequence  $\{W_i\}_{i=1}^k$ of Whitehead moves from $G_0$ to $G_k$ with corresponding composition
 $\widetilde{N}(W_k)\cdots \widetilde{N}(W_1)\in \Aut(\pi_1)$ equal to the identity.   By definition, this implies that $\mb{X}_{G_0}=\mb{X}_{G_k}$.  Using the previous two lemmas, there exists two sequences of Whitehead moves comprised solely of type $1$ or $2$ moves  connecting $G_0$ and $G_k$ respectively  to $C_{G_0}=C_{G_k}$.    The composition of the first such  sequence and the inverse of the second is equivalent modulo relations to $\{W_i\}_{i=1}^k$  since  there exists a unique element  of the Ptolemy groupoid connecting any two marked bordered fatgraphs.
~~~~~ \hfill{}
\end{proof}

\section{Chord slide algorithm and the image of $\widetilde{N}$}\label{chorddia}

In this section, we introduce an algorithm which produces a path in the mapping class groupoid from any bordered fatgraph to a fixed ``symplectic basepoint.'' As a consequence, we obtain an extension of the identity representation $id\colon MC(\S1g)\ra MC(\S1g)$ of the mapping class group.  Similarly in the next section, we will apply this algorithm in several guises to extend various representations.  

\subsection{Chord diagrams and the chord slide algorithm}
\label{chords}
We begin with an algorithm for linear chord diagrams described in \cite{barnatan} in terms of  ``chord slides''.   Let $C_G$ be the chord diagram associated to a bordered fatgraph $G$ and let $c$ and $d$ be two chords of $C_G$ so that the  endpoint $v(\mb{c})$ of $c$ immediately precedes  $v(\mb{d})$ in the left-to-right ordering along the core of $C_G$.   We define the \emph{slide} of $v(\mb{c})$ along $d$ to be the composition  of a Whitehead move on the edge $e$ of the core separating $v(\mb{c})$ and $v(\mb{d})$ followed by the Whitehead move on the chord $d$.   Similarly, we define the slide of $v(\mb{d})$ along $c$ to be the Whitehead move on $e$ followed by the Whitehead move on $c$.  Note that as the notation suggests, the result of the two moves is to slide the vertex along the boundary cycle so that it is adjacent to the opposite vertex of the chord along which it was slid.

 A marking of the bordered fatgraph $G$ induces a marking of the fatgraph $C_G$, and under a slide, the markings of all chords remain fixed except for the chord upon which the slide was performed.  For example, under the slide of  $v(\mb{c})$ along $d$ as discussed above, the marking of the oriented chord $\mb{d}$ changes  from $\mb{d}$  to $\mb{dc}$.  However, note that the effect on the linearly ordered set of generators $\mb{X}_G$ is more complicated as the ordering of the elements as well as their preferred orientations may change under such a slide.

\begin{figure}[!h]
\begin{center}
\epsffile{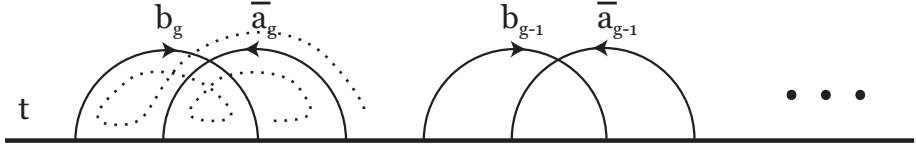}
\caption{Symplectic  chord diagram.}
\label{fig:sld}
\end{center}
\end{figure}

 Now, define the genus $g$ {\it symplectic chord diagram} to be the unique genus $g$ fatgraph $S$ such that $C_S=S$ and for any marking of $S$, $\mb{t}=\prod_{i=g}^{1}[\mb{x}_{2i},\mb{\bar x}_{2i-1}]$ with $\mb{X}_S=(\mb{x}_1, \ldots, \mb{x}_{2g})$.   We have depicted such a fatgraph in Figure \ref{fig:sld} where we have used the labels $\mb{b}_{i}=\mb{x}_{2i}$ and $\mb{\bar a}_{i}=\mb{x}_{2i-1}$ so that $\mb{t}=\prod_{i=g}^{1}[\mb{b}_{i},\mb{a}_{i}]$.  

 The \emph{chord slide algorithm} can now be described as follows.  Given a chord diagram $C_G$ associated to a fatgraph $G$, label the left-most chord of $C_G$ by $b_g$ and label the left-most chord which crosses $b_g$ by $a_g$ (note that such a chord must exist).  Next, sequentially slide all endpoints of chords (other than $b_g$ and $a_g$) which lie between the leftmost endpoint of $b_g$ and the rightmost endpoint of $a_g$ along the path represented by the dotted line in Figure \ref{fig:sld} so that all endpoints of chords lie to the right of $b_g$ and $a_g$.  Next, label the left-most chord appearing after $b_g$ and $a_g$ by $b_{g-1}$ and label the left-most chord which crosses $b_{g-1}$ by $a_{g-1}$.  Repeating this procedure, we eventually obtain a fatgraph isomorphic to $S$, cf. \cite{barnatan}.

\begin{theorem}\label{thm:mclift}
There is an explicit  extension 
\[
\widetilde{id}\colon {\mathfrak{Pt}}(\Sigma_{g,1})\ra MC(\Sigma_{g,1})
\]
to the Ptolemy groupoid
of the identity homomorphism of $MC(\Sigma_{g,1})$.

\end{theorem}
\begin{proof}
Consider a Whitehead move $W\colon G_1\ra G_2$ on a marked fatgraph $G_1$. 
  Let $S_1$ and $S_2$ be the respective marked symplectic chord diagrams obtained from $G_1$ and $G_2$ by performing the branch reduction algorithm followed by the chord slide algorithm.  Since $S_1$ and $S_2$ are isomorphic as unmarked fatgraphs, there exists a unique element $\varphi$ of $MC(\S1g)$ such that $\varphi(S_1)=S_2$, and we define  $\widetilde{id}(W)=\varphi$. 
  This gives a well-defined map $\widetilde{id}\colon  {\mathfrak{Pt}}(\Sigma_{g,1})\ra MC(\S1g)$ which extends the identity homomorphism by construction.  \hfill{}
\end{proof}

Note that if we fix a marking $S\hra \S1g$ for the symplectic chord diagram $S$, a modification of the proof actually provides a representation 
$\widetilde{id}\colon \mf{MC}(\Sigma_{g,1})\ra MC(\Sigma_{g,1})$ of the mapping class groupoid.
Also note that by considering the mapping class group $MC(\S1g)$ as a subgroup of $Aut(\pi_1)$, the theorem provides yet another extension of Nielsen's embedding;  
however, this extension has the disadvantage that it no longer depends on six essential cases.  
  
By combining Theorem \ref{thm:mclift} and the action of $MC(\S1g)$ on the first integral homology $H=H_1(\S1g,\bZ)$ of $\S1g$, we immediately obtain 

\begin{corollary}\label{uglycor}
There is an explicit canonical extension 
  \begin{equation}\label{eq:actiononH}
\tilde \tau_0\colon {\mathfrak{Pt}}(\S1g)\ra Sp(H).
  \end{equation}
of the symplectic representation of $MC(\S1g)$.
\end{corollary}

\subsection{The Image of $\widetilde{N}$}\label{sect:image}

We conclude  this  section by   describing the image of the extension of the Nielsen embedding.   This image cannot be all of  $\Aut(F_{2g})$ as the combinatorics of bordered fatgraphs put limitations on which sets of generators for $F_{2g}$ can arise from the greedy algorithm.  For example, due to the preferred orientation of edges, if $\mb{X}_G$ is a set of generators for $G$, then the set obtained from $\mb{X}_G$ by replacing $\mb{x}_i$ with $\mb{\bar x}_i$ for some $i$ cannot arise from a marked bordered fatgraph.

More generally, we have the following result, which implicitly describes the image of the extension of the Nielsen embedding.

\begin{proposition}\label{image}
After some number of replacements $\mb{x}\mapsto\mb{\bar x}$,  a set of generators $\mb{X}$  of $\pi_1$ arises as the set $\mb{X}_G$ of a generators for a marked bordered fatgraph $G$ if and only if the  element of $\pi_1$ representing $\partial\S1g$ can be written as a reduced word which contains each element $\mb{x}$ of  $\mb{X}$ and its inverse $\mb{\bar x}$ exactly once.  Moreover, the set of replacements $\mb{x}\mapsto\mb{\bar x}$ performed on $\mb{X}$ is uniquely determined and explicitly computable.
\end{proposition}

\begin{proof}
We employ a construction which is essentially the reverse of  Observation \ref{word} to build a  chord diagram from the word $w$ representing the class of the boundary in the letters $\mb{X}$.  Begin with a straight line segment  and $2g$ oriented chords labeled by $\mb{X}$ and then attach the ends of the edges to the line according  to the appearance of the corresponding letters in $w$ as in Observation \ref{word} to obtain a fatgraph $C$  with tail $\mb{t}$ (oriented pointing to the right).  The orientations of the chords $\mb{X}$ endow $C$ with a surjective $\pi_1$-marking such that $\pi_1(\mb{t})$ is the class of the boundary by construction. 

If $C$ has only one boundary cycle, then it  is a bordered genus $g$ fatgraph which endows the elements of  $\mb{X}$ with preferred orientations.  Thus, after replacing some $\mb{x}$ with their inverses according to their preferred orientations, we have realized $\mb{X}$ as the set of generators $\mb{X}_C$ for $C$  as required.

In order to derive a contradiction, now assume that the fatgraph $C$ has more than one boundary cycle.  By an Euler characteristic argument, this number must be odd, say $2n+1$ with $n>0$.    Note that the oriented chords of $C$ still endow $C$ with an abstract (but not geometric) surjective $\pi_1$-marking.  By the transitivity of Whitehead moves, there exists a sequence of moves which takes this fatgraph $C$ to a chord diagram $C'$ with tail with $2n$ isolated chords followed on the right by a genus $g-n$ symplectic chord diagram.  (See \cite{barnatan} for an explicit algorithm which is a generalization of  the chord slide algorithm, where the resulting diagram is called a ``$(2n,g-n)$-caravan''.)

 We again denote the tail of $C'$ by $\mb{t}$  since its value in $\pi_1$ remains fixed under any sequence of Whitehead moves.  If we then label the oriented chords of $C'$ (in their right-to-left appearance) by $\{\mb{\bar a}'_i,\mb{b}'_i\}_{i=1}^{g}$,  this provides a set of generators of $\pi_1$.  The contributions of the isolated chords $\{\mb{\bar a}'_i,\mb{b}'_i\}_{i=g-n+1}^{g}$ to the word representing $\mb{\bar t}$ in these letters  cancel so that  $\mb{\bar t}=\prod_{i=1}^{g-n}[\mb{a}'_i,\mb{b}'_i]$ as a word in these letters.  However, we can always find a  set of generators $\{\mb{\bar a}_i,\mb{b}_i\}_{i=1}^g$ for which  $\mb{\bar t}=\prod_{i=1}^{g}[\mb{a}_i,\mb{b}_i]=\prod_{i=1}^{g-n}[\mb{a}'_i,\mb{b}'_i]$.  By Proposition 6.8 of \cite{LS}, we must have $g-n\geq g$, a contradiction as required.
\end{proof}

\section{The symplectic   representation}

Just as for $\pi_1$-markings, the greedy algorithm applied to a geometrically $H$-marked bordered fatgraph $G$ results in a canonical linearly ordered basis $H(\mb{X}_G)$ of $H$.   We call  a basis of $H$ arising in this way for some marked bordered fatgraph $G$ a \emph{geometric basis} of $H$.  A geometric  basis $H(\mb{X}_G)=\{X_1,X_2, \ldots,X_{2g}\}$ has the property that  $X_i\cdot X_j$ equals -1 only if $i<j$, while it equals 1 only if $i>j$.  Thus, the intersection matrix of $\mb{X}_G$ is given by a skew symmetric $2g$-by-$2g$ matrix with only $0$'s and $1$'s below the diagonal.

 Fix a rank $2g$ symplectic vector space $(V,\omega)$.  Recall that a standard integral symplectic basis for  $(V,\omega)$ is a basis $\{A_i,B_i\}_{i=1}^{g}$ for $V$ such that the symplectic pairing $\omega$ takes values $\omega(A_i,B_j)=\delta_{ij}$ and $\omega(A_i,A_j)=\omega(B_i,B_j)=0$, for all $i,j$.
 
 While a standard symplectic basis of $H$ is not quite a geometric basis, any geometric $H$-marking of the symplectic chord diagram $S$ provides a geometric basis which differs from a symplectic one only in the signs of half of its elements.  In this way, any such basis provides a symplectic  isomorphism 
 $
 H\cong (V,\omega)
 $. By applying the branch reduction and chord slide algorithms, we thus obtain the following (cf. Corollary \ref{cor:explicitiso})
 \begin{corollary}\label{cor:explicitHiso}
For every $H$-marked bordered fatgraph $G$, there is  an explicit canonical integral symplectic basis for $H$, thus a canonical symplectic  isomorphism $H\cong (V,\omega)$.
\end{corollary}

  Given any two  symplectic bases $\mc{B}=\{A_i,B_j\}$ and $\mc{B}'=\{A'_i,B'_j\}$ of a symplectic vector space, the linear map taking $A_i\mapsto A'_i$ and $B_j\mapsto B'_j$  lies in $Sp(2g,\bZ)$.  Thus, completely analogously to Theorem \ref{thm:canAutlift} by combining Corollary \ref{cor:explicitHiso} and \eqref{eq:actiononH}, we obtain the following
  \begin{theorem}\label{thm:canHlift}
There is an explicit extension
\[
\hat{\tau}_0\colon \mf{MC}(\S1g)\ra Sp(2g,\bZ)
\]
to the mapping class groupoid 
of the symplectic representation 
of the mapping class group.
\end{theorem}

\subsection{The rational algorithm}

One may be interested to know if an extension of the symplectic representation with target  $Sp(2g,\bZ)$ can be obtained through more algebraic methods.  Here we describe such an approach which uses only linear algebra and the $H$-markings of bordered fatgraphs.  The new ingredient is  to provide a different but analogous isomorphism to that provided by Corollary \ref{cor:explicitHiso}.  While the following method  only works over the rationals for generic bases of $H$, it in fact is an integral algorithm for geometric bases since it can be realized by certain ``dual chord slides'' as shown in \cite{bene}.

Consider  an ordered geometric basis $H(\mb{X}_G)$ of $H$ and let $A_1 = X_1$.
Let $i\geq 2$ be minimal, such that $X_1\cdot
X_i \neq 0$, and renumber the $X_j$, for $j\geq 2$ by
interchanging $X_2$ and $X_i$. Let $b_1 = \frac1{X_1\cdot X_2}X_2$ and  define
$$X'_j = X_j - (X_j\cdot B_1)A_1 + (X_j\cdot A_1)B_1,$$
for $j\geq 3$.

By repeating this process on the 
ordered set $(X'_3,\ldots, X'_{2g})$ of independent vectors
in $H\otimes \bQ$, we eventually arrive at a symplectic basis of $H\otimes \bQ$.  By the result of \cite{bene}, this basis is in fact integral, and we have defined another $MC(\S1g)$-equivariant map from geometric to symplectic bases of $H$, thus also another extension of the symplectic representation.

\section{Other identity extensions}

In analogy to the extension of the identity representation given in Theorem \ref{thm:mclift}, we conclude by describing two other extensions of identity representations: one for the 
Torelli group $\mc{I}(\S1g)$ and one for the subgroup
$MC(\Lambda)$ of mapping classes preserving the Lagrangian $\Lambda< H$, i.e., $\Lambda$ is a maximal isotropic subspace.

\begin{theorem}\label{thm:Blift}
Given any geometric  basis  $\mc{B}$ for $H$, there is an explicit extension 
\[
\widetilde{id}_{\mc{B}} \colon {\mathfrak{Pt}}(\Sigma_{g,1})\ra \mc{I}(\S1g)
\]
 to the Ptolemy  groupoid of the identity homomorphism of the Torelli group, which is natural in the sense that
if $\phi\in MC(\S1g)$ then 
$$\widetilde{id}_{{\mathcal B}} (W:G\to G') = \phi^{-1}~\bigl [\widetilde{id}_{\phi({\mathcal B})}(\phi(W):\phi(G)\to\phi (G'))\bigr ] ~\phi.$$
\end{theorem}

\begin{theorem}\label{thm:Llift}
Given any  integral Lagrangian subspace $\Lambda$ of $H$, there is an explicit extension 
\[
\widetilde{id}_\Lambda\colon {\mathfrak{Pt}}(\Sigma_{g,1})\ra MC(\Lambda)
\]
to the Ptolemy groupoid of the identity homomorphism of $MC(\Lambda)$, which is natural in the sense
that  if $\phi\in MC(\S1g)$, then 
$$\widetilde{id}_{{\Lambda}} (W:G\to G') =\phi^{-1}~ \bigl [\widetilde{id}_{\phi (\Lambda)}(\phi(W):\phi(G)\to\phi (G'))\bigr ]
~\phi.$$
\end{theorem}

The proofs are quite similar and given or sketched in the next section after first developing the requisite tools here.  Just as the proof of Theorem \ref{thm:mclift} involved an algorithm which took any bordered fatgraph to a fixed ``symplectic basepoint'' bordered fatgraph, the above theorems are similarly based on an algorithm which takes any geometrically $H$-marked bordered fatgraph to a fixed  $H$-marked bordered fatgraph, i.e., a fixed ``symplectic basepoint'' in the Torelli groupoid.  Also as in Theorem \ref{thm:mclift}, if we fix a marked bordered fatgraph with corresponding geometric basis $\mc{B}$,  then Theorem \ref{thm:Blift} in fact leads to a representation of the Torelli groupoid in $\mc{I}(\S1g)$.  Similarly, fixing a marked bordered fatgraph,  the representation of Theorem \ref{thm:Llift} can also be extended to the Torelli groupoid.  

\subsection{Homology markings and chord slides}

Under a chord slide, the $H$-marking of a linear fatgraph evolves in a simple way: up to sign and permutation, all $H$-markings of chords are fixed except the one being slid over, which is modified by adding or subtracting the $H$-marking of the slid chord.  For example, consider the chord slide of Figure \ref{fig:moveexample}, where we begin with an isolated pair of overlapping chords with $H$-markings $B_i$ and $-A_i$.  
  When the left end of the $B_i$-marked chord is slid along the $-A_i$-marked chord, we obtain a new isolated pair of overlapping chords which are $H$-marked $A_i+B_i$ and $B_i$ as in the figure.   Thus, this chord slide corresponds to the transformation 
 \[
 A_i\mapsto -B_i, \quad B_i\mapsto A_i+B_i,
 \]
 which is easily seen to be a symplectic transformation.
   \begin{figure}[h!]
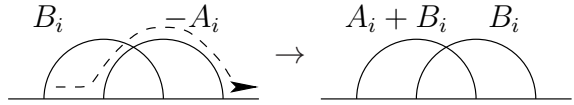

 \[
 \begin{array}{c}\input{moveii.pstex_t}\end{array}  \ra \begin{array}{c}\input{moveiiresult.pstex_t}\end{array}
 \]
 \caption{Evolution of $H$-marking under a chord slide.}\label{fig:moveexample}
 \end{figure}

 More generally, we have the 
\begin{lemma}
Assume that $C$ is a symplectic  chord diagram with chords $H$-marked by the  basis $\{-A_i,B_i\}$, $1\le i\le g$.  Then  the following elements of $Sp(2g,\bZ)$  can be realized in terms of chord slides ($1\le i\ne j\le g$):
 \begin{enumerate}
 \item[i$^{\pm}$)] $A_i\mapsto A_i \pm  B_i$.
  \item[ii$^{\pm}$)] $B_i\mapsto B_i  \pm A_i$.
 \item[iii)] $A_i\mapsto B_i\mapsto -A_i$.  
 \item[iv)] $A_i\mapsto A_j\mapsto A_i$, $B_i\mapsto B_j\mapsto B_i$, 
 \item[v$^{\pm}$)] $A_i\mapsto A_i \pm  A_{i+1}$, $B_{i+1}\mapsto B_{i+1}\mp B_i$.
 \end{enumerate}
\end{lemma}
\begin{proof}
The moves of types $i^+$ and $ii^+$ are provided by the following two chord slides 
 \[
  \begin{array}{c}\input{movei.pstex_t}\end{array} \begin{array}{c}\input{moveim.pstex_t}\end{array},
 \]
 and the moves of types $i^-$ and $ii^-$ are obtained similarly.  

The move of type $iii$ is obtained by combining moves of type $i^{\pm}$ and $ii^{\pm}$ with moves
similar to that illustrated in Figure \ref{fig:moveexample}.
The move of type $iv$ is provided by iterations of sequences of chord slides of the following type
 \[
 \begin{array}{c}\input{moveiii.pstex_t}\end{array}.
 \]
The move of type $v^+$ is provided by composing the following two sequences
 \[
 \begin{array}{c}\input{moveiv1.pstex_t}\end{array}~~~ \begin{array}{c}\input{moveiv2.pstex_t}\end{array}
 \]
 while type $v^-$ is similar.\hfill{}
\end{proof}

We now apply this lemma to prove the following
\begin{lemma}\label{lemA4}
There is an algorithm starting with any marked symplectic chord diagram $C\hra \S1g$ and any primitive integral vector  $v\in H$ (i.e., $v$ extends to an integral basis of $H$) which
produces a sequence of chord slides on $C$ resulting in a marked symplectic  chord diagram $C'$ with the leftmost chord $H$-marked by $v$.
\end{lemma}
\begin{proof}
Let $\mathcal{B}=\{A_i,B_i\}_{1\le i\le g}$ be the symplectic basis of $H$ given by the marking of $C$.  Since $v$ is integral, we have
\[
v=\sum_{i=1}^g  c_i A_i + d_i B_i, \quad c_i,d_i\in \bZ.
\]
By applying a sequence of type {\it iii} moves, we can assume that all $c_i,d_i\geq 0$.   

Next by  applying a sequence of type $i$ moves according to a ``homological division algorithm,'' we can obtain a new geometric basis $\mathcal{B'}$ such that either $c'_i=0$ or $d'_i=0$ for all $i=1,\ldots, g$.  For example, if $d_i\geq m_ic_i$, we would begin by applying a type $i^+$ move $m_i$ times to reduce the coefficient $d_i$ of $B_i$ by $m_i c_i$.  After completing this process, we can apply several type $iii$ moves to obtain a basis $\mc{B''}$ with all $c''_i=0$, so that

\[
v=\sum_{i=1}^g d''_i B''_i.
\]

Next by applying a similar division algorithm using type $iv$ and $v$ moves, we obtain a basis $\mathcal{B'''}=\{A'''_i,B'''_i\}_{1\le i\le g}$ with 
\[
v=d'''_i B'''_i,
\]
for some $i$.  
Since both $v$ and $B'''_i$ are integral basis elements, we must have $d'''_i=1$.  By applying type $iv$ moves, we can finally arrange that  $i=1$, as required. 
\end{proof}

\subsection{Proofs of Theorem \ref{thm:Blift} and \ref{thm:Llift}.}

For the proof of Theorem \ref{thm:Blift}, we devise an algorithm which will take any marked fatgraph $G$ to a symplectic chord diagram $C$  with corresponding geometric  $H$-basis given  by $\mc{B}$.  Once we have obtained such an algorithm, the theorem will follow analogously to the proof of Theorem \ref{thm:mclift}:  we compare the results of the algorithm for the marked fatgraphs $G$ and $G'$ which differ by a Whitehead move, and the difference in marking defines an element of $\mc{I}(\S1g)$.

 By applying the branch reduction and chord slide algorithms, we can assume that $G$ is a symplectic chord diagram and that $\mc{B}=\{-A_i,B_i\}$ corresponds to the symplectic basis $\{A_i,B_i\}$.  The algorithm then proceeds as follows.  First, we apply Lemma \ref{lemA4} using $v=B_{2g}$ to obtain a new symplectic chord diagram with leftmost chord labelled by $B_{2g}$.  It is easy to see that the unique chord overlapping with the leftmost one must be labelled by $-A_{2g}+kB_{2g}$, and by applying $k$ moves of type $i$,
we can arrange that the labeling is precisely $-A_i$.  We then apply this procedure to the genus $g-1$ symplectic chord diagram  subgraph with $v=B_{2g-1}$, and so on, until we arrive at a symplectic chord diagram  with geometric basis $\mc{B}$.  The naturality statement is a tautology tantamount to the existence of the algorithm.

The {proof of Theorem \ref{thm:Llift}} is similar and only sketched here.  The proof follows from an algorithm which takes any marked fatgraph $G$ to a symplectic  chord diagram $C$ with the property that the  Lagrangian subspace $\Lambda$ equals the span   of the $H$-markings of those chords of $C$ corresponding to the chords labelled $b_i$ produced in the chord slide algorithm.  
 The only truly new ingredient is the determination of a vector $v\in \Lambda$ in the application of Lemma \ref{lemA4}.  This is done by looking at the integral subspaces
$$\aligned
W_{2i-1} &=\Lambda \cap \textrm{span}(A_1,B_1, A_2,  \ldots, A_i),\\
W_{2i} &=  \Lambda \cap \textrm{span}(A_1,B_1, A_2,  \ldots, B_i).\\
\endaligned$$
For the minimal $j$ with $W_j$ nonempty,  the intersection is one-dimensional, hence contains a unique integral basis element $v\in W_j$.

\bibliographystyle{amsplain}

\begin{thebibliography}{999}

\bibitem{bamp}
J.\ E.\ Andersen, A.\ Bene, J.-B. Meilhan, R. C. Penner,
``Finite type invariants and fatgraphs'', preprint.


\bibitem{barnatan}
D. Bar-Natan, S.  Garoufalidis
{\it  On the Melvin-Morton-Rozansky conjecture},
 Invent. Math.  125  (1996)  no. 1, 103--133.
 
 \bibitem{bene}
A. Bene,
{\it A chord diagrammatic presentation of the mapping class group of a once-bordered surface}, to appear in Geom.  Dedicata., 
e-print: 0802.2747.


\bibitem{bkp}
A. Bene, N. Kawazumi, R.C. Penner,
{\it Canonical extensions of the Johnson homomorphisms to the Torelli groupoid},
Adv. Math. 221 (2009), 627--659.

\bibitem{fox}
R.H. Fox,
{\it Free differential calculus.I:Derivation in the free group ring},  Ann. of Math. (2)
57, 547Ð560, (1953).

\bibitem{LS} R. C. Lyndon, P. E. Schupp,
{\it Combinatorial Group Theory}, Classics in Mathematics, Springer-Verlag, Berlin (2001).


\bibitem{meyer}
W. Meyer
{\it  Die Signatur von Fl\"achenb\"undeln},
(German)  Math. Ann.  201  (1973) 239--264.

\bibitem{morita}
S. Morita,
{\it Abelian quotients of subgroups of the mapping class
     group of surfaces},
Duke Math. J. 70
(1993) 699--726.

\bibitem{moritapenner}
S. Morita, R.C. Penner
{\it Torelli groups, extended Johnson homomorphisms, and new cycles on the moduli space of curves},
Math. Proc. Camb. Phil. Soc. 144 (2008) 651--671.

\bibitem{nielsen}
J. Nielsen,
{\it Jakob Nielsen: collected mathematical papers},  Vol. 1.
Edited  by V.  Hansen,
Contemporary Mathematicians,  BirkhŠuser Boston Inc., MA  (1986).

\bibitem{penner}
R. Penner,
{\it The decorated Teichm\"uller space of punctured surfaces},
Comm. Math. Phys. 113
(1987) 299--339.



\bibitem{penner98}
---,
{\it Universal constructions in TeichmŸller theory},
 Adv. Math.  98  (1993)  143--215.

\bibitem{penner04}
---,
{\it Decorated Teichm\"uller theory of bordered surfaces},
Comm. Anal. Geom.   12
(2004) 793--820.

\bibitem{Suzuki02}
M. Suzuki, 
   {\it The Magnus representation of the Torelli group ${\mathcal{I}}_{g,1}$
   is not faithful for $g\geq2$},
  Proc. Amer. Math. Soc.  130
(2002)  909--914.  



\end{thebibliography}

\end{document}